\documentclass[11pt,reqno]{amsart}
\usepackage{latexsym}
\usepackage{amssymb}
\usepackage{amsthm}
\usepackage{amsmath}

\usepackage{graphicx}
\usepackage{color}

\allowdisplaybreaks   

\newtheorem{theorem}{Theorem}[section]

\newtheorem{corollary}[theorem]{Corollary}

\newtheorem{lemma}[theorem]{Lemma}

\newtheorem{definition}[theorem]{Definition}

\newtheorem{proposition}[theorem]{Proposition}
\newtheorem{remark}[theorem]{Remark}

\numberwithin{equation}{section} 

\makeatletter
\let\c@equation\c@theorem 
\makeatother

\DeclareMathOperator{\Hom}{Hom}

\DeclareMathOperator{\Ext}{Ext}

\DeclareMathOperator{\HH}{HH}
\DeclareMathOperator{\Ho}{H}
\DeclareMathOperator{\proj}{proj}
\newcommand{\g}{\mathfrak{g}}
\newcommand{\SL}{\mathfrak{sl}_2}
\newcommand{\Vg}{V(\SL)}
\newcommand{\bk}{\mathbf{k}}
\newcommand{\kG}{\bk G}
\newcommand{\N}{\mathcal{N}}

\newcommand{\K}{\mathcal{K}}
\newcommand{\I}{\mathcal{I}}

\begin{document}

%opening
\title[Finite Generation of Tate Cohomology of Symmetric Hopf Algebras]
{Finite Generation of Tate Cohomology \\
of Symmetric Hopf Algebras}

\author{Van C. Nguyen}
\address{Department of Mathematics\\Texas A\&M University \\College Station, TX 77843}
\email{vcnguyen@math.tamu.edu}

\begin{abstract}
Let $A$ be a finite dimensional symmetric Hopf algebra over a field $\bk$. 
We show that there are $A$-modules whose Tate cohomology is not finitely 
generated over the Tate cohomology ring $\widehat{\Ho}^*(A,\bk)$ 
of $A$. However, we also construct $A$-modules which have finitely generated 
Tate cohomology. It turns out that if a module in a connected component of 
the stable Auslander-Reiten quiver associated to $A$ has finitely generated 
Tate cohomology, then so does every module in that component. We apply some 
of these finite generation results on Tate cohomology to an algebra defined 
by Radford \cite{R} and to the restricted universal enveloping algebra of $\SL(\bk)$. 
\end{abstract}

\maketitle

%%%%%%%%%%%%%%%%%%%%%%%%%%%%%%%%%%%%%%%%%%%%
%%%%%%%%%%%%%%%%%%%%%%%%%%%%%%%%%%%%%%%%%%%%

\section{Introduction}
\label{sec:intro}

Many people have been interested in the finite generation of the cohomology 
of a finite dimensional Hopf algebra $A$. If such property holds, 
one can apply the theory of support varieties to the study of $A$-modules. 
It is known that there are several finite dimensional Hopf algebras whose 
cohomology over their base field $\bk$ is finitely generated, among them are: 
group algebras of finite groups, finite group schemes or equivalently 
finite dimensional co-commutative Hopf algebras, small quantum groups, 
and certain pointed Hopf algebras (see, for example, \cite[Introduction]{MPSW}). 
While the usual cohomology rings of such algebras are finitely generated, 
the same may not be true for their Tate cohomology rings. For example, 
it is shown in \cite{CCM} that the only finite groups $G$ having 
the property that every finitely generated $\kG$-module has finitely 
generated Tate cohomology have $p$-rank one or zero, where $p$ is the 
characteristic of the field $\bk$. The purpose of this paper is to 
investigate the finite generation property for Tate cohomology of 
a finite dimensional symmetric Hopf algebra $A$. If $M$ is a finitely 
generated $A$-module, we want to know whether $\widehat{\Ho}^*(A,M)$ is 
finitely generated as a graded module over $\widehat{\Ho}^*(A,\bk)$. 
While the methods we use here are mostly straightforward generalizations 
of those in \cite{CCM}, some additional assumption is needed to fit 
in the context. For instance, in Proposition \ref{Ben2p5.9.5}, we need 
$A$ to be a Hopf algebra so that tensor products of modules are again 
$A$-modules. Nonetheless, the author believes some of the results in 
this paper hold for finite dimensional symmetric $\bk$-algebras in 
general, not necessarily restricted to Hopf algebras. 

Throughout the paper, let $\bk$ be a field and $A$ be a finite dimensional 
symmetric Hopf algebra over $\bk$ with antipode $S$, coproduct $\Delta$, 
and augmentation $\varepsilon$. All modules are 
finitely generated left modules and tensor product is over $\bk$ unless 
stated otherwise. Via the augmentation map, $\bk$ is considered 
as a trivial $A$-module. We denote the usual cohomology ring 
of $A$ with coefficients in $\bk$ as $\Ho^*(A,\bk) :=\Ext_A^*(\bk,\bk)$. 

Let $M$ be a non-projective (indecomposable) $A$-module. 
The Tate cohomology of $A$ with coefficients in $M$ is: 
\[ \displaystyle \widehat{\Ho}^*(A,M):=\widehat{\Ext}^*_A(\bk,M) \cong 
\bigoplus_{n \in \mathbb{Z}} \underline{\Hom}_A(\Omega^n \bk, M) \cong 
\bigoplus_{n \in \mathbb{Z}} \underline{\Hom}_A(\bk, \Omega^{-n} M), \] 
where $\widehat{\Ext}$ is taken on an $A$-complete resolution of $\bk$, 
$\Omega$ is the Heller operator mapping an $A$-module to the kernel of 
a projective cover of that module, and $\underline{\Hom}_A$ 
is a quotient of $\Hom_A$ by homomorphisms that factor 
through a projective module. 
When $M$ is an $A$-bimodule and $A^e=A \otimes A^{op}$ is the enveloping 
algebra of $A$, the Tate-Hochschild cohomology of $A$ is:
\[ \displaystyle \widehat{\HH}^*(A,M):=\widehat{\Ext}^*_{A^e}(A,M) \cong 
\bigoplus_{n \in \mathbb{Z}} \underline{\Hom}_{A^e}(\Omega^n A, M) \cong 
\bigoplus_{n \in \mathbb{Z}} \underline{\Hom}_{A^e}(A, \Omega^{-n} M). \] 
In particular, by Theorem 7.2 in \cite{Nguyen}, $\widehat{\HH}^*(A,A)$ 
is isomorphic to $\widehat{\Ho}^*(A,A^{ad})$, where $A^{ad}$ is the left 
adjoint module of $A$. Using this relation, if the module $M=A^{ad}$ has 
the required hypotheses as in Sections \ref{sec:bdd fg submodules} and 
\ref{fg Tate cohomology}, then the corresponding finite generation results 
also hold for Tate-Hochschild cohomology of $A$. 

Since a symmetric algebra $A$ is isomorphic to its $\bk$-dual $\Hom_\bk(A,\bk)$ 
as $A-A$-bimodules, we obtain the Tate duality for symmetric algebras as a 
special case of Auslander-Reiten duality, see \cite[\S 2]{Lin}. We will 
use this Tate duality throughout our paper. Recall that 
for a finite dimensional Hopf algebra $A$, it is symmetric if and only if 
$A$ is unimodular and $S^2$ is inner \cite[\S 2.5]{Lorenz}. There are many 
finite dimensional symmetric Hopf algebras 
that are of interest, such as, group algebras of finite groups, 
commutative Hopf algebras (this includes the $\bk$-duals of co-commutative 
Hopf algebras), semisimple algebras, the Drinfield double of any Hopf algebra, 
the restricted universal enveloping algebra $V(\g)$ of a finite dimensional 
restricted $p$-Lie algebra $\g$ when $\g$ is nilpotent or semisimple, 
and an algebra defined by Radford in \cite{R}. 
Therefore, our finite generation of Tate cohomology 
results will add to the study of these algebras. 

The paper is organized as follows. In Section \ref{sec:bdd fg submodules}, 
we study $A$-modules whose Tate cohomology is not finitely generated. The key 
ingredients in this section are the boundedness conditions on finitely 
generated modules over Tate cohomology and the property that products 
in negative Tate cohomology of symmetric algebras are often zero \cite[\S 8]{Lin}. 
In Section \ref{fg Tate cohomology}, we study $A$-modules whose Tate cohomology 
is finitely generated. In particular, all modules in the connected component 
of the stable Auslander-Reiten quiver associated to $A$ which contains 
$\bk$ have this property. In Sections \ref{Radford's algebra} and 
\ref{V(sl(2,k))}, we apply some results from Section \ref{sec:bdd fg submodules} 
on the algebra constructed by Radford in \cite{R} and on the restricted 
universal enveloping algebra of $\SL$, respectively. By showing 
that these algebras have finitely generated usual cohomology 
but fail to do so for the Tate cohomology, these examples reassure 
the fact that finite generation behaves differently in negative 
cohomology. 

The author thanks S. Witherspoon for her insightful comments and support during 
the preparation of this manuscript. The author also thanks J. Pevtsova 
and her graduate student J. Stark for suggesting helpful references for 
the completion of Section \ref{V(sl(2,k))}.

%%%%%%%%%%%%%%%%%%%%%%%%%%%%%%%%%%%%%%%%%%%%
%%%%%%%%%%%%%%%%%%%%%%%%%%%%%%%%%%%%%%%%%%%%

\section{Modules with Bounds in Finitely Generated Submodules}
\label{sec:bdd fg submodules}
In this section, we show that there are $A$-modules whose Tate cohomology is 
not finitely generated. We recall some definitions and properties that were 
proved in \cite{CCM} for group algebras. The same proofs go through 
for any finite dimensional symmetric (Hopf) algebra $A$ over a field $\bk$. 
We present them here for completeness. 

\begin{definition}
\label{CCM3.1}
{\em A graded module $C=\bigoplus_{n \in \mathbb{Z}} C^n$ over 
$\widehat{\Ho}^*(A,\bk)$ has {\it bounded finitely generated submodules} if 
for any $m$, there exists a number $N=N(m)$ such that the submodule 
$D$ of $C$ generated by $\bigoplus_{n>m} C^n$ is contained 
in $\bigoplus_{n > N} C^n$. }
\end{definition}

\begin{lemma} 
\label{CCM3.2}
If a graded module $C=\bigoplus_{n \in \mathbb{Z}} C^n$ over 
$\widehat{\Ho}^*(A,\bk)$ has bounded finitely generated submodules 
and if $C^n \neq {0}$ for arbitrary small (meaning negative) 
values of $n$, then $C$ is not a finitely generated module over 
$\widehat{\Ho}^*(A,\bk)$. 
\end{lemma}

\begin{proof}
This follows immediately from the definition. Any finitely generated 
submodule of $C$ is contained in $\bigoplus_{n > N} C^n$ 
for some $N$, and hence, cannot generate all of $C$. 
\end{proof}

For a graded module $C=\bigoplus_{n \in \mathbb{Z}} C^n$, 
$C[s]=\bigoplus_{n \in \mathbb{Z}} C^{n+s}$ denotes a shift 
in $C$ by a degree $s$, for some integer $s$. 

\begin{lemma} 
\label{CCM3.4}
Suppose we have an exact sequence of $A$-modules: 
\[ 0 \rightarrow L \rightarrow M \rightarrow N \rightarrow 0 \]
which represents an element $\xi \in \Ext^1_A(N,L)$. Multiplication by $\xi$ induces a homomorphism 
$m_\xi: \widehat{\Ho}^*(A,N) \rightarrow \widehat{\Ho}^*(A,L)[1]$. 
Let $\K^*$ be the kernel of this map and $\I^*$ be the cokernel. 
Then we have an exact sequence of $\widehat{\Ho}^*(A,\bk)$-modules:
\[ 0 \rightarrow \I^* \rightarrow \widehat{\Ho}^*(A,M) \rightarrow \K^* \rightarrow 0. \]
Moreover, if $\K^*$ is not finitely generated over $\widehat{\Ho}^*(A,\bk)$, 
then neither is $\widehat{\Ho}^*(A,M)$.
\end{lemma}

\begin{proof}
By the naturality of the long exact sequence on Tate cohomology \cite[\S 3.2]{Nguyen}, we have:
\[ \cdots \xrightarrow{m_\xi} \widehat{\Ho}^n(A,L) \rightarrow 
\widehat{\Ho}^n(A,M) \rightarrow \widehat{\Ho}^n(A,N) 
\xrightarrow{m_\xi} \widehat{\Ho}^{n+1}(A,L) \rightarrow \cdots \]
The collection of the maps $m_\xi$ in the long exact sequence is a 
map of degree $1$ of $\widehat{\Ho}^*(A,\bk)$-modules
\[ m_\xi: \widehat{\Ho}^*(A,N) \rightarrow \widehat{\Ho}^*(A,L)[1]. \]
The last statement is a consequence of the fact that quotient modules of 
finitely generated modules are finitely generated. 
\end{proof}

Now for $d>0$, let $\xi$ be a non-zero element in $\widehat{\Ho}^d(A,k)$. Then 
$\xi$ is represented by a homomorphism $\xi: \Omega^d \bk \rightarrow \bk$. 
Let $L_\xi$ be the kernel of that map. If $\xi = 0$, we define 
$L_\xi := \Omega^d \bk \oplus \Omega \bk$. We have an exact sequence:
\[ 0 \rightarrow L_\xi \rightarrow \Omega^d \bk \xrightarrow{\xi} \bk \rightarrow 0. \] 

In the corresponding long exact sequence on Tate cohomology
\[ \cdots \rightarrow \widehat{\Ho}^{n-1}(A,\bk) \rightarrow \widehat{\Ho}^n(A,L_\xi) \rightarrow 
\widehat{\Ho}^n(A,\Omega^d \bk) \xrightarrow{\xi} \widehat{\Ho}^n(A,\bk) \rightarrow \cdots \]
the homomorphism labeled $\xi$ is multiplication by $\xi$. That is, it is the degree $d$ map
\[ \xi: \widehat{\Ho}^*(A,\bk)[-d] \rightarrow \widehat{\Ho}^*(A,\bk). \]
As a result, as in Lemma \ref{CCM3.4}, we have an exact 
sequence of $\widehat{\Ho}^*(A,\bk)$-modules: 
\[ 0 \rightarrow \I^*[-1] \rightarrow \widehat{\Ho}^*(A,L_\xi) \rightarrow \K^*[-d] \rightarrow 0 \]
where $\K^*$ and $\I^*$ are the kernel and cokernel of multiplication by $\xi$, respectively. 

\begin{lemma} 
\label{CCM3.5}
Suppose that $\xi \in \widehat{\Ho}^*(A,\bk)$, $d>0$, is a regular element 
on the usual cohomology ring $\Ho^*(A,\bk)$. Then
\begin{enumerate}
 \item $\K^t = {0}$ for all $t \geq 0$, and
 \item $\I^t = {0}$ for all $t <0$. 
\end{enumerate} 
\end{lemma}

\begin{proof}
Since $\xi$ is regular on $\Ho^*(A,\bk)$, it is clear that $\K^t = {0}$ for all $t \geq 0$. 
It remains to prove the second part of the lemma. We recall the Tate 
duality for symmetric algebras, see \cite[\S 2 and Lemmas 8.1, 8.2]{Lin}, 
equivalently, there is a natural nondegenerate bilinear form 
\[ \left\langle -,- \right\rangle: \widehat{\Ho}^{n-1}(A,\bk) \times \widehat{\Ho}^{-n}(A,\bk) \rightarrow \bk \]
such that $\left\langle \zeta\eta, \tau \right\rangle = \left\langle \zeta, \eta\tau \right\rangle$. 
For $t<0$, let $\alpha_1, \ldots, \alpha_r$ be a $\bk$-basis for $\widehat{\Ho}^{-t-1}(A,\bk)$. 
Because multiplication by $\xi: \widehat{\Ho}^{-t-1}(A,\bk) \rightarrow \widehat{\Ho}^{-t+d-1}(A,\bk)$ 
is a monomorphism by part (1), the elements $\xi\alpha_1, \ldots, \xi\alpha_r$ are linearly independent. 
So there must exist elements $\beta_1, \ldots, \beta_r$ in $\widehat{\Ho}^{t-d}(A,\bk)$ such that 
for all $i$ and $j$, we have:
\[ \left\langle \beta_i, \xi\alpha_j \right\rangle = \left\langle \beta_i\xi, \alpha_j \right\rangle = \delta_{ij} \]
where $\delta_{ij}$ is the usual Kronecker delta. Thus, the element 
$\beta_1\xi, \ldots, \beta_r\xi$ must be linearly independent and 
hence must form a basis for $\widehat{\Ho}^t(A,\bk)$. This implies that multiplication 
by $\xi: \widehat{\Ho}^{t-d}(A,\bk) \rightarrow \widehat{\Ho}^t(A,\bk)$ 
is a surjective map, for all $t<0$. Hence, its cokernel $\I^t = 0$. 
\end{proof}

There are examples of algebras for which all products in 
negative cohomology are zero. In particular, this holds 
for finite dimensional symmetric algebras whose usual 
cohomology has depth greater than or equal to $2$. 
Recall that by the discussion in \cite{Suarez-Alvarez}, the Tate 
cohomology of a Hopf algebra is always graded-commutative. 
Hence, a homogeneous regular sequence must automatically be 
central and when the characteristic of $\bk$ is not two, 
it must consist of elements in even degrees. Theorems 3.5 in 
\cite{BJO} and 8.3 in \cite{Lin} --- both are generalizations of 
the group cohomology result in \cite{BeCa} --- independently prove the following:

\begin{theorem}
\label{negative Tate products}
Let $A$ be a finite dimensional symmetric algebra over 
a field $\bk$. Let $M$ be a finitely generated $A$-module. 
If the depth of the usual cohomology (resp. Hochschild cohomology) 
of $M$ is two or more, then the Tate cohomology 
(resp. Tate-Hochschild cohomology) 
of $M$ has zero products in negative cohomology.
\end{theorem}

We show that using this property, for some $A$-module $M$, 
$\widehat{\Ho}^*(A,M)$ is not finitely generated.

\begin{proposition}
\label{CCM3.7}
Suppose $A$ is a finite dimensional symmetric Hopf algebra over 
a field $\bk$ and $\widehat{\Ho}^*(A,\bk)$ has the property 
that the products in negative cohomology are zero. If 
$\xi \in \Ho^d(A,\bk)$, $d>0$, is a regular element for $\Ho^*(A,\bk)$, 
then $\widehat{\Ho}^*(A,L_\xi)$ is not finitely generated 
as a module over $\widehat{\Ho}^*(A,\bk)$. 
\end{proposition}

\begin{proof}
Let $\K^*$ be the kernel of the multiplication by $\xi$ on $\widehat{\Ho}^*(A,\bk)$. 
By Lemma \ref{CCM3.5}, we have shown that $\K^*$ has elements 
only in negative degrees. Moreover, products of elements in negative 
degrees are zero by assumption. 
By \cite[Lemma 8.2]{Lin} or a direct generalization of 
\cite[Lemma 2.1]{BeCa} and by the fact that there is no bound 
on the dimensions of the spaces $\widehat{\Ho}^n(A,\bk)$ for 
negative values of $n$, it follows that $\K^*$ is not zero 
in infinitely many negative degrees. Thus, $\K^*$ has bounded 
finitely generated submodules and is not finitely generated over 
$\widehat{\Ho}^*(A,\bk)$ by Lemma \ref{CCM3.2}. It follows from 
Lemma \ref{CCM3.4} that $\widehat{\Ho}^*(A,L_\xi)$ is not 
finitely generated over $\widehat{\Ho}^*(A,\bk)$. 
\end{proof}

We say that a cohomology element $\xi \in \widehat{\Ho}^d(A,\bk)$ annihilates 
the Tate cohomology of a module $M$ if the cup product with $\xi$ is the zero operator 
on $\widehat{\Ext}_A^*(M,M)$. We generalize the proof of \cite[Prop. 5.9.5]{Ben2} 
to a finite dimensional (symmetric) Hopf algebra $A$. 

\begin{proposition}
\label{Ben2p5.9.5}
Let $A$ be a finite dimensional (symmetric) Hopf algebra over a field $\bk$. 
Suppose $M$ is a finitely generated $A$-module and 
$\xi \in \widehat{\Ho}^d(A,\bk)$, for some $d \in \mathbb{Z}$. 
Then $\xi$ annihilates $\widehat{\Ext}_A^*(M,M)$ if and only if 
\[ L_\xi \otimes M \cong \Omega(M) \oplus \Omega^d(M) \oplus (\proj), \]
where $(\proj)$ denotes some projective $A$-module. 
\end{proposition}

\begin{proof}
We note here that it is necessary for $A$ to be a Hopf algebra so that 
a tensor product of $A$-modules is again an $A$-module with action 
via the coproduct of $A$. By abuse of notation, 
let $\xi: \Omega^d \bk \rightarrow \bk$ be a cocycle representing the cohomology
element $\xi$. Let $L_\xi$ be its kernel. The proposition is obvious for $\xi=0$, 
as in this case, $L_\xi= \Omega^d \bk \oplus \Omega \bk$, and 
$\Omega^i(M) \cong \Omega^i \bk \otimes M \oplus (\proj)$ for any $i$. 

Assume $\xi \neq 0$. 
As before, we have an exact sequence:
\[ 0 \rightarrow L_\xi \rightarrow \Omega^d \bk \xrightarrow{\xi} \bk \rightarrow 0 \] 
By translating, we get the exact sequence:
\[ 0 \rightarrow \bk \rightarrow \Omega^{-1}(L_\xi) \rightarrow \Omega^{d-1} \bk \rightarrow 0 \]
that represents $\xi$ in $\widehat{\Ext}_A^1(\Omega^{d-1} \bk, \bk) \cong \widehat{\Ho}^d(A,\bk)$. 
Then we have that $\xi \cdot Id_M$ in $\widehat{\Ext}_A^d(M,M) \cong \widehat{\Ext}_A^1(\Omega^{d-1}(M),M)$ 
is represented by the sequence:
\[ 0 \rightarrow M \rightarrow \Omega^{-1}(L_\xi) \otimes M \rightarrow \Omega^{d-1} \bk \otimes M \rightarrow 0. \]
Now suppose $\xi$ annihilates $\widehat{\Ext}_A^*(M,M)$, then 
$\xi \cdot Id_M = 0$ and the above sequence splits. Hence, 
\[ \Omega^{-1}(L_\xi) \otimes M \cong M \oplus (\Omega^{d-1} \bk \otimes M), \] 
the middle term is the direct sum of the two end terms. Equivalently,
\[ \Omega^{-1} (L_\xi \otimes M) \cong M \oplus \Omega^{d-1}(M) \oplus (\proj). \]
Now translate everything by $\Omega$, we have:
\[ L_\xi \otimes M \cong \Omega(M) \oplus \Omega^d(M) \oplus (\proj). \]
Conversely, if $L_\xi \otimes M \cong \Omega(M) \oplus \Omega^d(M) \oplus (\proj)$, then the sequence 
\[ 0 \rightarrow \Omega(M) \rightarrow L_\xi \otimes M \rightarrow \Omega^d(M) \rightarrow 0 \]
splits. Translate everything by $\Omega^{-1}$, we get the sequence that represents $\xi \cdot Id_M$ also splits. 
Hence $\xi \cdot Id_M = 0$ and $\xi$ annihilates the Tate cohomology of $M$. 
\end{proof}

We are now ready to prove the main theorem of this section.

\begin{theorem} 
\label{nonfgTate}
Suppose $A$ is a finite dimensional symmetric Hopf algebra over a field $\bk$ and 
$\widehat{\Ho}^*(A,\bk)$ has the property that the products in negative cohomology 
are zero. Let $\xi \in \Ho^d(A,\bk)$, $d>0$, be a regular element and $M$ be
a finitely generated $A$-module such that $\widehat{\Ho}^*(A,M) \neq 0$. 
If for some $t>0$, $\xi^t$ annihilates the Tate cohomology of $M$ 
and of $L_{\xi^t}$, then $\widehat{\Ho}^*(A,M)$ is not 
finitely generated as a module over $\widehat{\Ho}^*(A,\bk)$. 
\end{theorem}

\begin{proof}
By assumption, $\widehat{\Ho}^*(A,M) \neq 0$, 
so by Lemma \ref{CCM3.2}, it is enough to show that 
$\widehat{\Ho}^*(A,M)$ has bounded finitely generated submodules. 

Now since $\xi^t$ annihilates the Tate cohomology of $M$ for some $t>0$, 
it follows from Proposition \ref{Ben2p5.9.5} that 
\[ L_{\xi^t} \otimes M \cong \Omega(M) \oplus \Omega^{dt}(M) \oplus (\proj). \]
Thus, $\widehat{\Ho}^*(A,M)$ has bounded finitely generated submodules if and only if 
$\widehat{\Ho}^*(A,L_{\xi^t} \otimes M)$ also has this property.

We first recall that for left $A$-modules $M$ and $N$, 
$\Hom_\bk(M,N)$ is a left $A$-module via the action: 
$(a \cdot f)(m)= \sum a_1 f(S(a_2)m)$, for $a \in A, m \in M$, 
and $f \in \Hom_\bk(M,N)$. When $N=\bk$ is the trivial 
$A$-module, the above action simplifies to the action of $A$ 
on $M^*:=\Hom_\bk(M,\bk)$: $(a \cdot f)(m)=f(S(a)m)$. 
Moreover, when $M$ and $N$ are finite dimensional as $\bk$-vector spaces, 
$\Hom_\bk(M,N) \cong N \otimes M^*$ as left $A$-modules, \cite[\S 2.1]{Lorenz}. 
We let $\widehat{\Ho}^*(A,\bk) \cong \widehat{\Ext}_A^*(\bk,\bk)$ act 
on $\widehat{\Ext}_A^*(M,M)$ via $- \otimes M$. 
By \cite[Cor. 3.1.6]{Ben1}, Proposition \ref{Ben2p5.9.5}, 
and the hypothesis that $\xi^t$ annihilates the Tate 
cohomology of $L_{\xi^t}$, we have:
\begin{eqnarray*}
\widehat{\Ext}_A^*(L_{\xi^t}, L_{\xi^t}) & \cong & \bigoplus_{n \in \mathbb{Z}} \underline{\Hom}_A(\Omega^n(L_{\xi^t}),L_{\xi^t}) \\
& \cong & \bigoplus_{n \in \mathbb{Z}} \underline{\Hom}_A(\Omega^n \bk \otimes L_{\xi^t},L_{\xi^t}) \\
& \cong & \bigoplus_{n \in \mathbb{Z}} \underline{\Hom}_A(\Omega^n \bk, \Hom_\bk(L_{\xi^t},L_{\xi^t})) \\
& \cong & \bigoplus_{n \in \mathbb{Z}} \underline{\Hom}_A(\Omega^n \bk, L_{\xi^t} \otimes (L_{\xi^t})^*) \\
& \cong & \bigoplus_{n \in \mathbb{Z}} \underline{\Hom}_A(\Omega^n \bk, L_{\xi^t} \otimes \Omega^{-dt-1}L_{\xi^t}) \\
& \cong & \bigoplus_{n \in \mathbb{Z}} \underline{\Hom}_A(\Omega^n \bk, \Omega^{-dt-1}(L_{\xi^t} \otimes L_{\xi^t})) \\
& \cong & \bigoplus_{n \in \mathbb{Z}} \underline{\Hom}_A(\Omega^n \bk, \Omega^{-dt-1}(\Omega L_{\xi^t} \oplus \Omega^{dt} L_{\xi^t} \oplus (\proj))) \\
& \cong & \bigoplus_{n \in \mathbb{Z}} \underline{\Hom}_A(\Omega^n \bk, \Omega^{-dt} L_{\xi^t} \oplus \Omega^{-1} L_{\xi^t}) \\
& \cong & \widehat{\Ho}^*(A, \Omega^{-dt} L_{\xi^t} \oplus \Omega^{-1} L_{\xi^t})
\end{eqnarray*}
where $(L_{\xi^t})^* = \Hom_\bk(L_{\xi^t},\bk) \cong \Omega^{-dt-1}L_{\xi^t}$ 
by a generalization of \cite[Prop. 11.3.3]{CTVZ}.

As $\xi$ is a regular element, it is not hard to see that $\xi^t$ is also a regular element. 
By a similar argument as in Proposition \ref{CCM3.7}, we have that 
$\widehat{\Ext}_A^*(L_{\xi^t}, L_{\xi^t})$ has bounded finitely 
generated submodules. By definition, there exists a number $N$ such that
\[ \widehat{\Ho}^*(A,\bk) \cdot \bigoplus_{n \geq m} \widehat{\Ext}_A^n(L_{\xi^t},L_{\xi^t}) 
\subseteq \bigoplus_{n \geq N} \widehat{\Ext}_A^n(L_{\xi^t},L_{\xi^t}). \]

Now let $m$ be any integer. Let 
\[ \mathcal{N} := \bigoplus_{n \geq m} \widehat{\Ho}^n(A,L_{\xi^t} \otimes M). \] 
We observe that the action of 
$\widehat{\Ho}^*(A,\bk)$ on $\widehat{\Ho}^*(A,L_{\xi^t} \otimes M)$ 
via $- \otimes L_{\xi^t} \otimes M$ factors through the map 
$\widehat{\Ho}^*(A,\bk) \rightarrow \widehat{\Ext}_A^*(L_{\xi^t}, L_{\xi^t})$, 
and the target of that map has bounded finitely generated submodules.  
Thus, we have:
\begin{eqnarray*}
 \widehat{\Ho}^*(A,\bk) \cdot \mathcal{N} & \subseteq & \widehat{\Ho}^*(A,\bk) \cdot 
 \left(\bigoplus_{n \geq m} \widehat{\Ext}_A^n(L_{\xi^t},L_{\xi^t})\right)\left(\bigoplus_{n \geq m} \widehat{\Ho}^n(A,L_{\xi^t} \otimes M)\right) \\
 & \subseteq & \left(\bigoplus_{n \geq N} \widehat{\Ext}_A^n(L_{\xi^t},L_{\xi^t})\right)\left(\bigoplus_{n \geq m} \widehat{\Ho}^n(A,L_{\xi^t} \otimes M)\right) \\
 & \subseteq & \bigoplus_{n \geq m+N} \widehat{\Ho}^n(A,L_{\xi^t} \otimes M).
\end{eqnarray*}
Therefore, $\widehat{\Ho}^*(A,L_{\xi^t} \otimes M)$ has bounded 
finitely generated submodules, and so does $\widehat{\Ho}^*(A,M)$. 
\end{proof}

\begin{remark}{\em   %{\em....} is to undo the format, eg. italic, of what inside the bracket to the opposite format of what existing.
Suppose that the cohomology in even degrees $\Ho^{ev}(A,\bk)$ is 
finitely generated (so it is a finitely generated commutative algebra, 
since $\Ho^*(A,\bk)$ is graded-commutative) 
and for any finite dimensional $A$-module $M$, the $\Ho^{ev}(A,\bk)$-module 
$\Ext_A^*(M,M)$ is finitely generated. Then under this assumption, 
one can define the support varieties for modules as follows: 

Let $I_A(M,M)$ be the annihilator of the action of $\Ho^{ev}(A,\bk)$ on 
$\Ext_A^*(M,M)$, a homogeneous ideal of $\Ho^{ev}(A,\bk)$, and let 
$\mathcal{V}_A(M):=\mathcal{V}_A(M,M)$ denote the maximal ideal spectrum 
of the finitely generated commutative $\bk$-algebra $\Ho^{ev}(A,\bk)/I_A(M,M)$. 
As the ideal $I_A(M,M)$ is homogeneous, the variety $\mathcal{V}_A(M)$ 
is conical and is called {\it the support variety of $M$}. 

Then the hypothesis ``for some power $\xi^t$ of $\xi$, 
$\xi^t$ annihilates the Tate cohomology of $M$ and of $L_{\xi^t}$'' 
in Theorem \ref{nonfgTate} can be translated as 
$\mathcal{V}_A(M) \subseteq \mathcal{V}_A\left\langle \xi \right\rangle$ and
$\mathcal{V}_A(L_{\xi^t}) \subseteq \mathcal{V}_A\left\langle \xi \right\rangle$, 
where $\mathcal{V}_A\left\langle \xi \right\rangle$ is 
the support variety of the ideal generated by $\xi$. }
\end{remark}

%%%%%%%%%%%%%%%%%%%%%%%%%%%%%%%%%%%%%%%%%%%%%
%%%%%%%%%%%%%%%%%%%%%%%%%%%%%%%%%%%%%%%%%%%%%

\section{Modules with finitely generated Tate cohomology}
\label{fg Tate cohomology}

It is obvious that any module $M$ which is a direct sum of Heller translates $\Omega^i \bk$ 
has finitely generated Tate cohomology, as in this case, its Tate cohomology is a 
direct sum of copies of $\widehat{\Ho}^*(A,\bk)$:
\begin{eqnarray*}
 \widehat{\Ho}^*(A,M) & \cong & \bigoplus_{n \in \mathbb{Z}} \underline{\Hom}_A(\Omega^n \bk,M) 
 \cong \bigoplus_{n \in \mathbb{Z}} \underline{\Hom}_A(\Omega^n \bk, \bigoplus_i \Omega^i \bk) \\
 & \cong & \bigoplus_i \bigoplus_{n \in \mathbb{Z}} \underline{\Hom}_A(\Omega^n \bk, \Omega^i \bk) \\
 & \cong & \bigoplus_i \bigoplus_{n \in \mathbb{Z}} \underline{\Hom}_A(\Omega^{n-i} \bk, \bk)
 \cong \bigoplus_i \widehat{\Ho}^*(A,\bk).
\end{eqnarray*} 
In this section, we show that in general, there are more modules 
with this property. First, we consider the Tate cohomology of 
module $M$ which can occur as the middle term of an exact 
sequence of the form:
\[ 0 \rightarrow \Omega^m \bk \rightarrow M \rightarrow \Omega^n \bk \rightarrow 0 \]
for some $m, n \in \mathbb{Z}$. Such a sequence represents an element $\xi$ in 
\[ \widehat{\Ext}_A^1(\Omega^n \bk, \Omega^m \bk) \cong \widehat{\Ext}_A^{n+1-m}(\bk,\bk) \cong \widehat{\Ho}^{n+1-m}(A,\bk). \]
Without lost of generality, we can apply the shift operator 
$\Omega^{-m}$ and assume that the sequence has the form
\[ 0 \rightarrow \bk \rightarrow M \rightarrow \Omega^n \bk \rightarrow 0 \] 
for some $n$, and that $\xi \in \widehat{\Ho}^{n+1}(A,\bk)$.

\begin{theorem}
\label{fgTate}
Suppose that for the module $M$ and cohomology 
element $\xi$ as above, the map $\xi: \widehat{\Ho}^*(A,\bk) \rightarrow \widehat{\Ho}^*(A,\bk)$ 
given by multiplication by $\xi$ has a finite dimensional image. Suppose that the usual cohomology 
ring $\Ho^*(A,\bk)$ is Noetherian. Then the Tate cohomology $\widehat{\Ho}^*(A,M)$ is 
finitely generated as a module over $\widehat{\Ho}^*(A,\bk)$.
\end{theorem}

\begin{proof}
As in Lemma \ref{CCM3.4}, we have an exact sequence of $\widehat{\Ho}^*(A,\bk)$-modules: 
\[ 0 \rightarrow \I^* \rightarrow \widehat{\Ho}^*(A,M) \rightarrow \K^*[-n] \rightarrow 0 \]
for $\xi \in \widehat{\Ho}^{n+1}(A,\bk)$, $\K^*$ is the kernel of multiplication by $\xi$ 
on $\widehat{\Ho}^*(A,\bk)$ and $\I^*$ is the cokernel. By hypothesis, the image of 
multiplication by $\xi$ has finite total dimension. Hence, in all but a finite number of 
degrees $i$, multiplication by $\xi$ is the zero map. Clearly, $\I^*$ is finitely 
generated over $\widehat{\Ho}^*(A,\bk)$. So, $\widehat{\Ho}^*(A,M)$ is finitely 
generated over $\widehat{\Ho}^*(A,\bk)$ if and only if $\K^*$ has the same property. 

View $\K^*$ as a module over the usual cohomology ring $\Ho^*(A,\bk)$. The elements in 
non-negative degrees form a submodule $\mathcal{L}^*= \sum_{m \geq 0} \K^m$, which is 
finitely generated over $\Ho^*(A,\bk)$, since $\Ho^*(A,\bk)$ is Noetherian by assumption. 

Let $\mathcal{M}^*$ be the $\widehat{\Ho}^*(A,\bk)$-submodule of $\K^*$ generated by  $\mathcal{L}^*$. 
We want to show that $\mathcal{M}^* = \K^*$ therefore proving the finite generation of $\K^*$. 
For all $m \geq 0$, $\K^m \subseteq \mathcal{M}^*$ by construction. It remains to show 
$\K^m \subseteq \mathcal{M}^*$ for all $m<0$. 

Because the quotient of $\widehat{\Ho}^*(A,\bk)$ by $\K^*$ is finite dimensional, we must 
have that $\widehat{\Ho}^j(A,\bk) = \K^j$ for sufficiently large $j$. For some sufficiently 
large $j$, we can find an element $0 \neq \gamma \in \K^j$ which is a regular element for the 
usual cohomology ring $\Ho^*(A,\bk)$. By a generalized version of Lemma 3.5 in \cite{BeCa}, 
we know that multiplication by $\gamma$ is a surjective map:
\[ \gamma: \widehat{\Ho}^{m-j}(A,\bk) \rightarrow \widehat{\Ho}^m(A,\bk) \]
whenever $m<0$. Hence, for all $m<0$, we must have $\widehat{\Ho}^{m-j}(A,\bk)\gamma = 
\K^m$. Since $\gamma \in \mathcal{M}^*$, we get that $\K^m \subseteq \mathcal{M}^*$ 
for all $m<0$. Therefore, $\K^* = \mathcal{M}^*$ is finitely generated as a module over 
$\widehat{\Ho}^*(A,\bk)$. This proves the theorem. 
\end{proof}

\begin{remark}
{\em  %{\em....} is to undo the format, eg. italic, of what inside the bracket to the opposite format of what existing.
There are many examples of sequences satisfying the condition of the theorem above. 
In particular, it is often the case that multiplication by an element $\xi$ in negative 
degree has a finite dimensional image. An example is the element in degree $-1$ which 
represents the almost split sequence for the module $\bk$. In addition, if the depth of 
$\Ho^*(A,\bk)$ is two or more, then all products in negative cohomology are zero; and 
the principal ideal generated by any element in negative cohomology contains no non-zero 
elements in positive degrees, for example, by \cite[Lemma 8.2]{Lin} or 
a direct generalization of \cite[Lemma 2.1]{BeCa}. 
Hence, multiplication by any element $\xi$ in negative 
cohomology has a finite dimensional image.}
\end{remark}

\begin{corollary}
The middle term of the almost split sequence 
\[ 0 \rightarrow \Omega^2 \bk \rightarrow M \rightarrow \bk \rightarrow 0 \]
ending with $\bk$ has finitely generated Tate cohomology. 
\end{corollary}

\begin{proof}
The almost split sequence corresponds to an element $\xi \in \widehat{\Ho}^{-1}(A,\bk)$. 
One of the defining properties of the almost split sequence is that for any module $N$, 
the connecting homomorphism $\delta$ in the corresponding sequence 
\[ \cdots \rightarrow \underline{\Hom}_A(N,M) \rightarrow \underline{\Hom}_A(N,\bk) \xrightarrow{\delta} 
\widehat{\Ext}^1_A(N, \Omega^2 \bk) \rightarrow \cdots \]
is non-zero if and only if $N \cong \bk$ \cite[Prop. V.2.2]{ARS}. This connecting 
homomorphism is multiplication by $\xi$. Now any element $\eta \in \widehat{\Ho}^d(A,\bk)$ 
is represented by a map $\eta: \Omega^d \bk \rightarrow \bk$. Therefore, we have $\xi\eta = 0$ 
whenever $d \neq 0$. This implies that multiplication by $\xi$ on $\widehat{\Ho}^*(A,\bk)$ has 
a finite dimensional image, and it follows from the above theorem that 
$\widehat{\Ho}^*(A,M)$ is finitely generated as a module over $\widehat{\Ho}^*(A,\bk)$. 
\end{proof}

\begin{remark}
{\em  %{\em....} is to undo the format, eg. italic, of what inside the bracket to the opposite format of what existing.
For an almost split sequence $0 \rightarrow P \rightarrow Q \rightarrow R \rightarrow 0$, 
we have $R \cong \text{TrD}(P)$, equivalently $P \cong \text{DTr}(R)$, 
see \cite[Prop. V.1.14]{ARS}. But for symmetric algebras, 
$\text{TrD} \cong \Omega^{-2}$ and $\text{DTr} \cong \Omega^2$. 
So for any indecomposable non-projective module $N$ over a symmetric algebra, 
Theorem V.1.15 in \cite{ARS} shows the existence of an almost split sequence 
\[ 0 \rightarrow \Omega^2 N \rightarrow M \rightarrow N \rightarrow 0. \] }
\end{remark}

\begin{proposition}
Let $N$ be a finitely generated indecomposable non-projective $A$-module that is not isomorphic 
to $\Omega^i \bk$ for any $i$. Consider the almost split sequence 
\[ 0 \rightarrow \Omega^2 N \rightarrow M \rightarrow N \rightarrow 0 \]
ending in $N$. If $N$ has finitely generated Tate cohomology, then so does the middle term $M$.
\end{proposition}

\begin{proof}
For any $i \in \mathbb{Z}$, the connecting homomorphism $\delta$ in the corresponding sequence
\[ \cdots \rightarrow \underline{\Hom}_A(\Omega^i \bk,M) \rightarrow \underline{\Hom}_A(\Omega^i \bk,N) \xrightarrow{\delta} 
\widehat{\Ext}^1_A(\Omega^i \bk, \Omega^2 N) \rightarrow \cdots \]
is zero because $\Omega^i \bk \ncong N$. Hence, $\delta$ induces the zero map on Tate cohomology. 
So the long exact sequence in Tate cohomology breaks into short exact sequences:
\[ 0 \rightarrow \widehat{\Ho}^*(A,\Omega^2 N) \rightarrow \widehat{\Ho}^*(A,M) \rightarrow \widehat{\Ho}^*(A,N) \rightarrow 0. \]
It follows that if $\widehat{\Ho}^*(A,N)$ is finitely generated, then $\widehat{\Ho}^*(A,M)$
is also finitely generated. 
\end{proof}

Combining the last two results, we have the following theorem:
\begin{theorem}
\label{CCM5.4}
If a module in a connected component of the stable Auslander-Reiten 
quiver associated to $A$ has finitely generated 
Tate cohomology, then so does every module in that component. 
In particular, all modules 
in the connected component of the quiver which contains 
$\bk$ have finitely generated Tate cohomology.
\end{theorem}

%%%%%%%%%%%%%%%%%%%%%%%%%%%%%%%%%%%%%%%%%%%%%
%%%%%%%%%%%%%%%%%%%%%%%%%%%%%%%%%%%%%%%%%%%%%

\section{Example: Radford's Algebra}
\label{Radford's algebra}

The following Hopf algebra $A$ is taken from \cite[Example 1]{R}. 
Let $N>1$ and $\bk$ be a field that contains a primitive $N$-th root of unity $\omega$. 
Let $A$ be an algebra generated over $\bk$ by 
elements $x, y, g$ subject to the relations:
\[ g^N=1, \hspace{0.5cm} x^N=y^N=0, \hspace{0.5cm} xg=\omega gx, \hspace{0.5cm} gy=\omega yg, \hspace{0.5cm} xy=\omega yx. \]
$A$ is of dimension $N^3$ and has the Hopf structure:
\begin{align*}
 \Delta(g) &= g \otimes g, &\varepsilon(g) &= 1, &S(g) &=g^{-1},\\
 \Delta(x) &= x \otimes g + 1 \otimes x, &\varepsilon(x) &= 0, &S(x) &=-xg^{-1}, \\
 \Delta(y) &= y \otimes g + 1 \otimes y, &\varepsilon(y) &= 0, &S(y) &=-yg^{-1}.
\end{align*}
It is shown in \cite{R} that $A$ is unimodular and $S^2$ is an inner 
automorphism. Hence, $A$ is symmetric. 

Let $Y=yg^{-1}$. Using the above relations, one can check that $x$ and $Y$ commute. 
Consider a subalgebra $B$ of $A$ generated by $x$ and $Y$ subject to the 
following relations:
\[ x^N = Y^N = 0, \hspace{0.5cm} xY=Yx. \]
In particular, $B$ is the truncated polynomial algebra which can be 
considered as the quantum complete intersection 
$\bk[x,Y]/(x^N, Y^N, xY-Yx) \cong \bk[x]/(x^N) \otimes \bk[Y]/(Y^N)$. 
Using the K\"{u}nneth Theorem, the cohomology of $B$ can be obtained by 
tensoring together the cohomology of $\bk[x]/(x^N)$ and the cohomology of $\bk[Y]/(Y^N)$. 
One can also construct a $B$-projective resolution of $\bk$ by taking the 
tensor product of the following periodic resolutions: 
\[ \cdots \xrightarrow{\cdot x^{N-1}} \bk[x]/(x^N) \xrightarrow{\cdot x} \bk[x]/(x^N) \xrightarrow{\cdot x^{N-1}} \bk[x]/(x^N) \xrightarrow{\cdot x} \bk[x]/(x^N) \xrightarrow{\varepsilon_x} \bk \rightarrow 0, \]
and 
\[ \cdots \xrightarrow{\cdot Y^{N-1}} \bk[Y]/(Y^N) \xrightarrow{\cdot Y} \bk[Y]/(Y^N) \xrightarrow{\cdot Y^{N-1}} \bk[Y]/(Y^N) \xrightarrow{\cdot Y} \bk[Y]/(Y^N) \xrightarrow{\varepsilon_Y} \bk \rightarrow 0,\]
where $\varepsilon_x(x)=0$ and $\varepsilon_Y(Y)=0$. This construction 
has been done in the literature, for example in \cite[\S 4]{MPSW}. 
Using the relations $xg=\omega gx$ and $gy=\omega yg$, 
we can see that Radford's algebra $A = B \# \kG$, 
where $G = \left\langle g \right\rangle$ acts on $B$ by automorphisms 
for which $x, Y$ are eigenvectors:
\[ gxg^{-1} = \;^g x = \omega^{-1} x, \hspace{0.5cm} \;^g Y = \omega Y. \]
Given basis elements $\{1_x,x\}$ of $\bk[x]/(x^N)$ 
and $\{1_Y,Y\}$ of $\bk[Y]/(Y^N)$, for $b=1_x, 1_Y,x$, or $Y$, 
we define the action of $g$ on the above resolutions as:
\begin{itemize}
 \item In degree $2i$, $g \cdot b := (^gb)$.
 \item In degree $2i+1$, $g \cdot b := \begin{cases} 
                                  \omega^{-1}(^gb), &\quad b=1_x,x \\
                                  \omega(^gb), &\quad b=1_Y,Y.
                                 \end{cases}$
\end{itemize}
One checks that this group action commutes with the differentials in each degree. 

The cohomology ring $\Ho^*(B,\bk)$ is generated by $\xi_j, \eta_i$, 
for $i,j \in \left\{1,2\right\}$, where deg$(\xi_j)=2$ and deg$(\eta_i)=1$, 
subject to the following relations: 
\[ \xi_1\xi_2=\xi_2\xi_1, \hspace{0.5cm} \eta_1\eta_2 = -\eta_2\eta_1, \hspace{0.5cm} \eta_i\xi_j=\xi_j\eta_i, \hspace{0.5cm} (\eta_i)^2=0, \]
(see, for example, \cite[Theorem 4.1]{MPSW}). 
Note that $\xi_1, \eta_1$ generate $\Ho^*(\bk[x]/(x^N),\bk)$ and 
$\xi_2, \eta_2$ generate $\Ho^*(\bk[Y]/(Y^N),\bk)$. 
If $N=2$, $\xi_i$ is a scalar multiple of $\eta_i^2$. 
As $A = B \# \kG$ and the characteristic of $\bk$ 
does not divide the order of $G$, we have:
\[ \Ho^*(A,\bk) \cong \Ho^*(B,\bk)^G, \]
the invariant ring under the above $G$-action defined at the chain level. 
By (4.2.1) in \cite{MPSW}, the induced action of $G$ 
on generators $\xi_j, \eta_i$ is given by: 
\[g \cdot \xi_j= \xi_j, \hspace{0.5cm} g \cdot \eta_1 = \omega \eta_1, \hspace{0.5cm} g \cdot \eta_2 = \omega^{-1} \eta_2. \]
Thus, $\Ho^*(A,\bk) \cong \bk[\xi_1, \xi_2]$, where deg$(\xi_j)=2$. 

The elements $\xi_1, \xi_2$ form a regular sequence on $\Ho^*(A,\bk)$. 
In fact, the depth of $\Ho^*(A,\bk)$ is $2$. By Theorem \ref{negative Tate products}, 
the Tate cohomology $\widehat{\Ho}^*(A,\bk)$ of $A$ has zero products 
in negative cohomology. Therefore, since each $\xi_j$ is a regular element 
on $\Ho^*(A,\bk)$, it follows from Proposition \ref{CCM3.7} that 
$\widehat{\Ho}^*(A,L_{\xi_j})$ is not finitely generated 
as a module over $\widehat{\Ho}^*(A,\bk)$, for $j=1,2$.

%%%%%%%%%%%%%%%%%%%%%%%%%%%%%%%%%%%%%%%%%%%%%
%%%%%%%%%%%%%%%%%%%%%%%%%%%%%%%%%%%%%%%%%%%%%

\section{Example: The restricted enveloping algebra of $\SL(\bk)$}
\label{V(sl(2,k))}

Let $\bk$ be an algebraically closed field of characteristic $p>3$. 
Let $\g:=\SL(\bk)$ be the restricted $p$-Lie algebra of $ 2 \times 2$ 
trace-zero matrices over $\bk$. As a $\bk$-module, 
$\g$ is generated by: 
\[e=\begin{pmatrix}
  0 & 1 \\
  0 & 0
  \end{pmatrix}, 
\hspace{0.5cm} f=\begin{pmatrix}
  0 & 0 \\
  1 & 0
  \end{pmatrix}, 
\hspace{0.5cm} h=\begin{pmatrix}
  1 & 0 \\
  0 & -1
  \end{pmatrix}, \] 
with the Lie algebra structure: 
\[ [h,f]=-2f, \hspace{0.5cm} [h,e]=2e, \hspace{0.5cm} [e,f]=h, \]
and the map $[p]: \g \rightarrow \g$ is given by:
\[ e^{[p]}=f^{[p]}=0, \hspace{0.5cm} h^{[p]}=h. \]

Let $V(\g)$ be the restricted enveloping algebra of $\g$. 
It is defined as the quotient algebra: 
\[ V(\g) := T(\g) / \left\langle X \otimes Y - Y \otimes X - [X,Y], X^{\otimes p} - X^{[p]}, X, Y \in \g \right\rangle, \]
equivalently,
\[ V(\g) = U(\g) / \left\langle X^{\otimes p} - X^{[p]}, X \in \g \right\rangle, \]
where $T(\g)$ is the tensor algebra and $U(\g)$ is 
the universal enveloping algebra of $\g$. 
$V(\g)$ is a finite dimensional, co-commutative Hopf algebra 
over $\bk$ with the Hopf structure: 
\[ \Delta(X) = X \otimes 1 + 1 \otimes X, \hspace{0.5cm} \varepsilon(X)=0, \hspace{0.5cm} S(X)=-X, \]
for all $X \in \g$. 
A {\it restricted $\g$-module} is a module $M$ of $\g$ on which $X^{[p]}$ 
acts as the $p$-th iterate of $X$, for any $X \in \g$. The category 
of restricted $\g$-modules is equivalent to the category of $V(\g)$-modules. 
From here, all $\g$-modules are assumed to be restricted and 
will be referred to as $V(\g)$-modules. 

To a restricted Lie algebra $\g$, we associate the {\it nullcone}
$\N=\N(\g)$ of $\g$, which is the closed subvariety of $\g$ consisting 
of all nilpotent elements. We also define the {\it restricted nullcone}
of $\g$ to be the subvariety 
\[ \N_1(\g) = \left\{X \in \g \; | \; X^{[p]}=0 \right\} \]
of $[p]$-nilpotent elements in $\g$. 

The cohomology $\Ho^*(V(\g),M)$ of $V(\g)$ with coefficients 
in a restricted $\g$-module $M$ is defined as the 
cohomology of the augmented algebra $V(\g)$ over $\bk$. 
There is a close relationship between the nullcone $\N(\g)$ and the 
cohomology $\Ho^*(V(\g),\bk)$ as described in the following theorem:

\begin{theorem}[\cite{FP1}] 
Let $G$ be a simple, simply connected algebraic group over an 
algebraically closed field $\bk$ of characteristic $p>0$. Assume 
that $G$ is defined and split over the prime field $\mathbb{F}_p$. 
Let $\g$ be the Lie algebra of $G$. 
If $p$ is greater than the Coxeter number of $G$, then 
there is a $G$-equivariant isomorphism of algebras: 
\[ \Ho^*(V(\g),\bk) \cong \bk[\N]^{(1)}, \] 
where $\bk[\N]$ is the coordinate ring of the nullcone $\N$ of $\g$, 
and $\bk[\N]^{(1)}$ means $\bk[\N]$ to be regarded as a $G$-module by 
composing the usual conjugation action of $G$ on $\bk[\N]$ 
with the Frobenius morphism $f: G \rightarrow G$. 
\end{theorem}

Recall that the usual basis for $\g=\SL(\bk)$ is $\left\{e,f,h\right\}$. 
Let $\left\{x,y,z\right\}$ be its dual basis. As an affine space, 
$\SL(\bk)$ can be identified with $\mathbb{A}^3$ and has coordinate 
ring $\bk[x,y,z]$. Since every $2 \times 2$ nilpotent matrix squares to $0$, 
we have 
\[ \begin{pmatrix}
  z & x \\
  y & -z
  \end{pmatrix}^2 = (xy+z^2)\begin{pmatrix}
  1 & 0 \\
  0 & 1
  \end{pmatrix} \] 
is zero whenever $xy+z^2$ is zero. Hence, independent of $p$, 
the nullcone $\N(\SL)$ is a quadric $\mathbb{A}^3$ defined 
by the equation $xy+z^2=0$. By the above theorem, 
\[ \Ho^*(\Vg,\bk) \cong \bk[\N(\SL)]^{(1)} \cong \bk[x,y,z]/(xy+z^2) \] 
is concentrated in even degrees as a graded ring. One can 
see that $\Ho^*(\Vg,\bk)$ has depth $2$. Moreover, 
by \cite{Schue}, $\Vg$ is symmetric. Therefore, we can apply 
Theorem \ref{negative Tate products} to conclude that 
the Tate cohomology $\widehat{\Ho}^*(\Vg,\bk)$ has zero products 
in negative cohomology. It then follows from Proposition \ref{CCM3.7} 
that for each regular element $\xi$ of $\Ho^*(\Vg,\bk)$, 
the Tate cohomology $\widehat{\Ho}^*(\Vg,L_{\xi})$ is not finitely generated 
as a module over $\widehat{\Ho}^*(\Vg,\bk)$.

%%%%%%%%%%%%%%%%%%%%%%%%%%%%%%%%%%%%%%%%%%%%%
%%%%%%%%%%%%%%%%%%%%%%%%%%%%%%%%%%%%%%%%%%%%%

\end{document}